\documentclass[12pt]{amsart}

\usepackage[ansinew]{inputenc}
\usepackage{amsmath}
\usepackage{amsfonts}
\usepackage{amssymb}
\usepackage{amsthm}
\usepackage{graphicx}
\usepackage{latexsym}
\usepackage{color}
\usepackage{multirow}
\usepackage{multicol}
\usepackage{enumerate}
\usepackage[all]{xy}
\usepackage{url}
\usepackage{varioref}

\DeclareGraphicsExtensions{.pdf,.png}

\newtheoremstyle{mystyle}
  {}
  {}
  {}
  {}
  {\bfseries}
  {}
  {\newline}
  {}

\newtheorem{thm}{Theorem}[section]
\newtheorem{lem}[thm]{Lemma}
\newtheorem{cor}[thm]{Corollary}
\newtheorem{prop}[thm]{Proposition}

\theoremstyle{mystyle}

\newtheorem{ex}{Example}

\newcommand{\C}{\mathbb{C}}

\newcommand{\Q}{\mathbb{Q}}

\newcommand{\Z}{\mathbb{Z}}

\newcommand{\graf}[1]{\left\lbrace#1\right\rbrace}

\DeclareMathOperator{\Gor}{Gor}

\DeclareMathOperator{\id}{Id}

\DeclareMathOperator{\h}{h}

\DeclareMathOperator{\Rk}{Rk}

\DeclareMathOperator{\Dim}{Dim}

\DeclareMathOperator{\Tr}{Tr}
\DeclareMathOperator{\Stab}{Stab}

\DeclareMathOperator{\Aut}{Aut}

\DeclareMathOperator{\PGL}{PGL}

\DeclareMathOperator{\age}{age}
\DeclareMathOperator{\SL}{SL}

\DeclareMathOperator{\Fix}{Fix}
\DeclareMathOperator{\Sing}{Sing}
\DeclareMathOperator{\e}{e}

\renewcommand{\P}{\mathbb{P}}

\renewcommand{\H}{H}
\renewcommand{\S}{S}

\makeindex

\title[]{Calabi-Yau quotients with terminal singularities}

\author{Filippo F. Favale}
\address[Filippo F. Favale]{Department of Mathematics, University of Trento, via Sommarive 14,
I-38123 Trento, Italy}
\email{filippo.favale@unitn.it}

\subjclass[2010]{14J32, 14J50}

\thanks{This research project was partially supported by FIRB 2012 ''Moduli spaces and Applications" and by GNSAGA-INdAM.}

\begin{document}

\begin{abstract}
\noindent In this paper we are interested in quotients of Calabi-Yau threefolds with isolated singularities. In particular, we analyze the case when $X/G$ has terminal singularities. We prove that, if $G$ is cyclic of prime order and $X/G$ has terminal singularities, then $G$ has order lower than or equal to $5$. 
\end{abstract}

\maketitle

\section*{Introduction}
A Calabi-Yau variety is a projective variety $X$ that has trivial canonical bundle and no non-zero holomorphic $p-$forms for $1\leq p\leq \Dim(X)-1$. Even for threefolds, the ones concerning us in this paper, very little is known about many geometric aspects. For example, a topological classification is very far to be understood and all possible sets of Betti numbers of a Calabi-Yau threefold are not known. Here we are interested in automorphisms of a Calabi-Yau variety. Most of the interesting results in this area are for particular families of Calabi-Yau varieties. For example, Wilson has proved that a Calabi-Yau manifold whose second Chern class is positive on the K\"ahler cone has a finite automorphism group (see \cite{Wil}). In \cite{Ogu}, it is shown that the same is true if $X$ is a Calabi-Yau threefold with Picard number $2$. In this paper, we would like to say something about groups $G$ of automorphism of Calabi-Yau threefold that give quotients with terminal singularities.
\vspace{4mm}

\noindent The main tool we will use is the holomorphic Lefschetz fixed point formula. If $g$ is an automorphism of a complex threefold and if the fixed points of $g$ are isolated, this formula, in its basic form, gives a relation between the traces of $g^*$ restricted on $\H^{0,k}(X)$ and some contributions that depend only on the fixed points and the local actions in a neighbourhood of the fixed points of $g$. 
First of all, we will use the Lefschetz formula to find conditions that have to be satisfied by prime-order automorphisms with isolated fixed points. In Theorem \ref{THM:TERM} we will see that, if we allow only terminal singularities for the quotient $X/\Z_p$, then $p$ has to be equal to $2,3$ or $5$. Also, we are able to tell precisely the number of the fixed points for each case. 
\vspace{4mm}

\noindent We will then focus on small automorphisms of order $2,3$ or $5$. When $g$ has order $2$, a description of the quotient $X/\Z_2$ for all the possible values of $\Dim(\Fix(g))$ is given. It is shown that $\Dim(\Fix(g))$ has always pure dimension and that the quotients with terminal singularities are also the ones with isolated singularities. When the order of $g$ is $3$, a discussion on the number of fixed points for $\Fix(g)$ of dimension $0$ is given. Basically, one can only have $9$ fixed points if $X/\Z_3$ has terminal singularities or an even number of fixed points otherwise, i.e. when $X/\Z_3$ has Gorenstein singularities and it is a singular Calabi-Yau threefold. If $g$ has order $5$, the number of fixed points for the case $\Dim(\Fix(g))=0$ is studied when $g$ is not symplectic, i.e., $g^*|_{\H^{0,3}(X)}\neq \id$. It is shown that the minimal number of fixed points is achieved only if $X/\Z_5$ is terminal.
\vspace{4mm}

\noindent Finally, we will present some examples of automorphisms giving cyclic quotients with terminal singularities to show that each $p\in\{2,3,5\}$ can, in fact, occur. An example of a terminal quotient $X/(\Z_2\oplus\Z_2)$ is given. Two quotient $X/G$ with Gorenstein non isolated singularities (one with $G$ of order $32$ and the other with respect to $\Z_2$) are also investigated.

\section{General facts}
\noindent Let $X$ be a smooth Calabi-Yau threefold and let $g$ be an automorphism  of $X$ such that $o(g)=p$ is prime. Denote by $G$ the group generated by $g$ and assume that $\Dim(\Fix(g))=0$, i.e., $g$ has only isolated fixed points. Under this assumption, we can apply the holomorphic Lefschetz fixed point formula to $g$, namely:
\begin{equation}
\label{EQ:HLF}
\sum_{P\in \Fix{g}}\frac{1}{\det(\id-d_Pg)}=\sum_{k=0}^3(-1)^k\Tr\left(g^*|_{\H^{0,k}(X)}\right),
\end{equation}
where $d_Pg$ is the differential map induced by $g$ on $T_PX$. We will investigate the restriction given by this identity under some assumptions on the types of singularities of $X/G$.
\vspace{2mm}

\noindent First of all, let's try to understand the contribution given by one point to the right hand side of (\ref{EQ:HLF}). In a neighbourhood of a fixed point $P$, the action of $g$ on $X$ can be described in terms of a linearization of $g$, i.e., the action of $d_Pg:T_PX\rightarrow T_{g(P)}X=T_PX$. The automorphism $g$ has finite order, so we can diagonalize $d_Pg$, thus obtaining
$$d_Pg\leftrightsquigarrow\begin{bmatrix} \omega^{a_1(P)} & 0 & 0\\ 0 & \omega^{a_2(P)} & 0 \\ 0 & 0 & \omega^{a_3(P)} \end{bmatrix},$$
where $\omega=\e^{2\pi i/p}$ and $0\leq a_i(P)\leq p-1$ are the exponents determined up to permutation. The local equations for the fixed locus are $z_i(1-\omega^{a_i(P)})=0$; if the fixed point is isolated $a_i(P)>0$. Call $s(P)=a_1(P)+a_2(P)+a_3(P)$. 
\vspace{2mm}

\noindent Recall that a Calabi-Yau threefold has $\h^{0,1}(X)=\h^{0,2}(X)=0$ and trivial canonical bundle. Hence there exists an everywhere non vanishing holomorphic $3-$form which will be denoted by $\eta$. Moreover, $\H^{0,3}(X)=\eta\cdot \C$, so the action of an automorphism $g$ on $\H^{0,3}(X)$ is simply the multiplication by an element of $\C^*$ that is a root of unity of order $o(g)$.
Therefore, the right-hand side of (\ref{EQ:HLF}) is simply $1-\Tr\left(g^*|_{\H^{0,3}(X)}\right)=1-\omega^{r}$ for some $r$. 
We will call an automorphism of a Calabi-Yau threefold such that $g^*|_{\H^{0,3}(X)}=\id$ a {\it symplectic automorphism}. The set $\S(X)$ of such automorphisms is easily proven to be a normal subgroup of $\Aut(X)$.
\vspace{2mm}

\noindent As the following lemma shows, if $X$ is a Calabi-Yau threefold and if we know the local action around a fixed point we can obtain information about the action of $g$ on $\H^{0,3}(X)$. 

\begin{lem}
\label{LEM:LTG}
If $X$ is a Calabi-Yau threefold and $g\in \Aut(X)$, then one of the following holds:
\begin{itemize}
\item $\Fix(g)$ is empty and $g\in \S(X)$;
\item $P\in \Fix(g)$ and there exists an action of $g$ on the stalk $\Omega_{X,P}^3$; the action of $g$ is the multiplication by $\det{d_Pg}=\omega^{s(P)}$.
\item One has $\Tr\left(g^*|_{\H^{0,3}(X)}\right)=\omega^{s(P)}$ for every $P\in \Fix(g)$.
\end{itemize}
\end{lem}

\begin{proof}
If $\Fix(g)$ is empty, the holomorphic Lefschetz fixed point formula gives the equation $0=1-\omega^r$, so $r=0$, and this is equivalent to asking that $g \in \S(X)$.
Now, if there is a fixed point $P$, let's prove that the action of $g$ on the stalk of the canonical sheaf is given by $\det(d_Pg)$. If $\rho_P(g^*)$ is the map given by $g$ on the stalk over $P$ of the sheaf $\Omega_{X}^3$, we have the relation $\wedge^3 (d_Pg)^*=\rho_P(g^*)$. But $d_Pg$ is a linear automorphism of $T_PX$ whose dimension is $3$. This implies that $\wedge^3 (d_Pg)$ is the multiplication by $\det(d_Pg)$, thus proving the claim. The last part is an easy consequence of the second. 
\end{proof}

\noindent Let $X$ be a complex threefold and consider a fixed point $P$ of $g\in \Aut(X)$. The {\it age} of P with respect to the primitive root $\lambda$ of order $o(g)$ is $$\age(P,\lambda):=(a_1+a_2+a_3)/o(g),$$ where $\lambda^{a_i}$ are the eigenvalues of $d_Pg$ and $0\leq a_i\leq o(g)-1$. Recall that if $V$ is a vector space and $f:V\rightarrow V$ is linear, $f$ is a {\it quasi-reflection} if $\Rk(f-\id)=1$. A group $G$ acting on a complex manifold is said to be {\it small} if for every $g\in G$ and every $P\in \Fix(g)$ one has that $d_Pg:T_PX\rightarrow T_PX$ is not a quasi-reflection. This condition is equivalent to asking that $\Fix(G)$ has codimension at least $2$.
The following theorem recalls some well known facts about some types of singularities (see, for example, \cite{Mor}).
\begin{thm}
\label{THM:SUNTO}
Let $X$ be a complex threefold and consider a small group $G$ that acts on $X$. Call $\pi:X\rightarrow X/G$ the projection on the quotient and $G_P=\Stab_{G}(\{P\})$ the isotropy of $P$. Then $\Sing(X/G)=\Fix(G)/G$ and 
\begin{itemize}
\item $\pi(P)$ is a Gorenstein singularity if and only if $d_Pg\in  \SL(T_PX)$ for each $g\in G_P$;
\item $\pi(P)$ is a canonical singularity if and only if $\age(g,\lambda)\geq 1$ for each  primitive $\lambda$ and for each $g\in G_P$;
\item $\pi(P)$ is a terminal singularity if and only if $\age(g,\lambda)> 1$ for each primitive $\lambda$ and for each $g\in G_P$;
\item $\pi(P)$ is a terminal singularity if and only if $\det(d_Pg)$ is an eigenvalue of $d_Pg$ for each $g\in G_P$.
\end{itemize}
\end{thm}

\subsection{Quotients with terminal singularities}
\noindent From now on, $X$ is a Calabi-Yau threefold. Here we are interested in small cyclic groups of prime order which give a quotient with terminal singularities. The following calculation are inspired by the one done in \cite{Sob} for a Fano threefold.
\vspace{2mm}

\noindent Assume that $g$ is an automorphism of a Calabi-Yau threefold with a finite number of fixed point and with prime order $p$. If $P\in \Fix(g)$, consider the sum
\begin{equation}
S_n(P)=\sum_{\substack{0\le k_1,k_2,k_3\le p-1 \\ a_1(P)k_1+a_2(P)k_2+a_3(P)k_3\equiv_p n}} k_1k_2k_3,
\end{equation}
where $\omega^{a_i(P)}$ are the eigenvalues of $d_Pg$ and $\omega$ is a $p-$th  primitive root of unity.

\begin{thm}
\label{THM:CONTI}
Let $X$ be a Calabi-Yau threefold and $g$ an automorphism of prime order $p$ with only isolated fixed point. Let $0\leq r\leq p-1$ such that $$\Tr\left(g^*|_{\H^{0,3}(X)}\right)=\omega^r.$$ 
Then, the equality
\begin{equation}
\sum_{P\in\Fix(g)}\left(\frac{p^3(p-1)^3}{8}-pS_0(P)\right)=\left\lbrace
\begin{matrix}
p^4 & \mbox{ if } r\neq 0 \\
0 & \mbox{ if } r=0
\end{matrix}
\right.
\end{equation}
holds.
\end{thm}
\begin{proof}
Starting from the Lefschetz fixed point formula, we have
\begin{equation}
\label{EQ:LEFHOL}
\sum_{P\in \Fix(G)}\frac{1}{\det(I-d_Pg)}=:\Lambda(X,g)=1-\omega^{r}
\end{equation} 
\vspace{2mm}

\noindent If, as before, we call $\omega$ a root of unity and we let $\omega^{a_i(P)}$ be the eigenvalues of $d_Pg$ one has
$$\sum_{P\in \Fix(G)}\frac{1}{\det(I-d_Pg)}=\sum_{P\in \Fix(G)}\frac{1}{(1-\omega^{a_1(P)})(1-\omega^{a_2(P)})(1-\omega^{a_3(P)})}.$$
If $\lambda$ is a primitive root of unity, the following relation holds 
$$\frac{1}{1-\lambda}=-\frac{1}{p}\sum_{k=1}k\lambda^{k}.$$
Every $\omega^{a_i(P)}$ is a primitive root. Indeed, the fixed locus has dimension $0$, hence $a_i\not\equiv_p 0$. We can then write
$$\sum_{P\in \Fix(G)}\frac{1}{\det(I-d_Pg)}=-\frac{1}{p^3}\sum_{P\in \Fix(G)}\prod_{i=1}^3\left(\sum_{k_i}k_i\omega^{k_ia_i(P)}\right)=$$
$$=-\frac{1}{p^3}\sum_{P\in \Fix(G)}\sum_{0\leq k_1,k_2,k_3\leq p-1}k_1k_2k_3\omega^{k_1a_1(P)+k_2a_2(P)+k_3a_3(P)}.$$
The coefficient of $\omega^n$ (seeing $\Q(\omega)$ as a $\Q$ vector space with the standard basis) given by $P$ is precisely $-\frac{1}{p^3}S_n(P)$; so we can rewrite the left hand side of the holomorphic Lefschetz formula in the following way:
$$-\frac{1}{p^3}\sum_{P\in \Fix(G)}\sum_{n=0}^{p-1}\left(\sum_{\substack{0\le k_1,k_2,k_3\le p-1 \\ a_1(P)k_1+a_2(P)k_2+a_3(P)k_3\equiv_p n}} k_1k_2k_3\right)\omega^n=$$
$$=-\frac{1}{p^3}\sum_{P\in \Fix(G)}\sum_{n=0}^{p-1}S_n(x)\omega^n.$$
We end up with the following formula:
\begin{equation}
\label{EQ:LAMBDA}
\Lambda(X,g)=\sum_{x\in \Fix(G)}\frac{1}{\det(I-d_xg)}=-\frac{1}{p^3}\sum_{P\in \Fix(G)}\sum_{n=0}^{p-1}S_n(P)\omega^n.
\end{equation}
From equations (\ref{EQ:LEFHOL}) and (\ref{EQ:LAMBDA}) we obtain
$$\sum_{P\in \Fix(G)}\sum_{n=0}^{p-1}S_n(P)\omega^n=p^3\omega^r-p^3$$
$$\sum_{n=0}^{p-1}\left(\sum_{P\in \Fix(G)}S_n(P)\right)\omega^n+p^3-p^3\omega^r=0.$$
It is useful to separate the case $r=0$ and $r\neq 0$. Call $B_n$ the coefficient of $\omega^n$ in the left hand side of the last equation, namely
$$\begin{array}{c|c}
r\neq0 & r=0 \\
B_n:=\sum_{P\in \Fix(G)}S_n(P)+\left\lbrace\begin{matrix}
0 & \mbox{ se } n\neq 0,r \\
p^3 & \mbox{ se } n=0 \\
-p^3 & \mbox{ se } n=r
\end{matrix}\right. & 
B_n:=\sum_{P\in \Fix(G)}S_n(P) \\
\end{array}$$
This relation $$\sum_{n=0}^{p-1}B_n\omega^n=0$$ is true if and only if the coefficients $B_n$ are all equal. From $B_0=B_1=\dots=B_{p-1}$ one has $B_n-B_0=0$ for all $1\leq n\leq p-1$. Hence 
$$\sum_{n=1}^{p-1}(B_n-B_0)=0\Longleftrightarrow \sum_{n=1}^{p-1}(B_n)-(p-1)B_0=0\Longleftrightarrow \sum_{n=0}^{p-1}(B_n)-pB_0=0$$
If we solve for $B_n$, we have
$$\begin{array}{c|c}
r\neq0 & r=0 \\
\sum_{P}\Bigl(\bigl(\sum_{n=0}^{p-1}S_n(P)\bigl)-pS_0(P)\Bigl)=p^4 &
\sum_{P}\Bigl(\bigl(\sum_{n=0}^{p-1}S_n(P)\bigl)-pS_0(P)\Bigl)=0.
\end{array}$$
The quantity $S_n(x)$ is the sum of the product $k_1k_2k_3$ for which $a_1(P)k_1+a_2(P)k_2+a_3(P)k_3\equiv_p n$; hence the sum of the $S_n(x)$ for $0\leq n\leq p-1$ is simply the sum of $k_1k_2k_3$ for $0\leq k_i\leq p-1$. 
This is equivalent to the third power of the sum of the first $p-1$ integers. In the end, this can be written as
$$\sum_{P}\Bigl(\sum_{n=0}^{p-1}S_n(P)-pS_0(P)\Bigl)=\sum_{P}\biggl(\sum_{0\leq k_i\leq p-1}k_1k_2k_3-pS_0(P)\biggl)=$$
$$=\sum_{P}\biggl(\frac{p^3(p-1)^3}{8}-pS_0(P)\biggl).$$
This is the statement of the Theorem.
\end{proof}

\begin{lem}
\label{LEM:CONTOSOB}
Let $p$ be a prime and let $1\leq a,b,c\leq p-1$ such that $$((a+b+c) \mod p)\in\graf{a,b,c}.$$ Then
$$\sum_{\substack{0\leq k_1,k_2,k_3\leq p-1 \\ ak_1+bk_2+ck_3\equiv_p 0}} k_1k_2k_3=\frac{p}{2}\left[\frac{p^2(p-1)^2}{4}-\frac{p(p-1)(2p-1)}{6}\right].$$
\end{lem}
\begin{proof}
See \cite{Sob}.
\end{proof}

As a consequence of Lemma (\ref{LEM:CONTOSOB}), for every point $P\in \Fix(g)$ with 
$$((a_1(P)+a_2(P)+a_3(P)) \mod p)\in\graf{a_i(P)}$$
we have that $S_0(P)$ depends only on $p$, the order of the group.

\begin{thm}
\label{THM:TERM}
Assume that $X$ is a Calabi-Yau threefold and $g$ is an automorphism of prime order $p$ with at most isolated fixed points. Call $G:=<g>$ and $q$ the number of fixed point of $g$. If $X/G$ has at most terminal singularities, then one of the following holds:
\begin{itemize}
\item $G$ acts freely on $X$ ($q=0$) and the action is symplectic;
\item $G$ has fixed points and the action of $G$ is not symplectic; Moreover we have $p\in\graf{2,3,5}$.
\end{itemize}
If the second case occurs, $g$ has $16,9$ or $5$ fixed point if $p=2,3$ or $5$,  respectively.
\end{thm}

\begin{proof}
The quotient $X/G$ has at most terminal singularities, so $G$ is small because its fixed locus is either empty or it has dimension $0$. 
If $q=0$ then, by Lemma (\ref{LEM:LTG}), $G$ is symplectic and, by definition it acts freely on $X$.
Assume that there are fixed points. Using Lemma (\ref{LEM:CONTOSOB}) one has
$$\sum_{x\in\Fix(g)}\left(\frac{p^3(p-1)^3}{8}-pS_0(x)\right)=$$ $$=\sum_{x\in\Fix(g)}\left(\frac{p^3(p-1)^3}{8}-p\frac{p}{2}\left[\frac{p^2(p-1)^2}{4}-\frac{p(p-1)(2p-1)}{6}\right]\right)=$$
$$=q\left(\frac{p^3(p-1)^3}{8}-p\frac{p}{2}\left[\frac{p^2(p-1)^2}{4}-\frac{p(p-1)(2p-1)}{6}\right]\right)=q\frac{p^3(p^2-1)}{24}$$
which is equal to
$$
\left\lbrace
\begin{matrix}
p^4 & \mbox{ if } r\neq 0 \\
0 & \mbox{ if } r=0,
\end{matrix}
\right.$$
by Theorem (\ref{THM:CONTI}).
In the end
$$q\frac{(p^2-1)}{24}=
\left\lbrace
\begin{matrix}
p & \mbox{ if } r\neq 0 \\
0 & \mbox{ if } r=0
\end{matrix}
\right.$$
$G$ is a subgroup of $\S(X)$ if and only if $r=0$ but this implies $q=0$. So, if there are fixed points, the action of $G$ is not symplectic. If we assume to have fixed points (and hence singular points on the quotient), we have that their number is given by $$q=24p/(p^2-1).$$ The only values of $p$ for which $q$ is a positive integer are $2,3$ and $5$, for which $q$ is $16,9$ and $5$, respectively.
\end{proof}

\noindent As an easy consequence of the last theorem we have the following

\begin{cor}
Assume that $X$ is a Calabi-Yau threefold and that $G\leq \Aut(X)$ is a small group such that $X/G$ has at most terminal singularities. Then $|G|=2^a3^b5^c.$
\end{cor}


\section{Automorphisms of order $p\in \{2,3,5\}$}

\noindent In Section (\ref{SE:EX}) we will give some examples of Calabi-Yau threefolds with quotients having terminal singularities. Now, we will analyse the cases for which $p\in\{2,3,5\}$.

\subsection{Small involutions}
\noindent Here we are interested in small involutions. For this case, we will not restrict to the case of fixed locus of dimension $0$.

\begin{prop}
\label{PROP:INVDIM1}
Let $X$ be a Calabi-Yau threefold and let $g$ be a small involution with fixed points. The following are equivalent:
\begin{enumerate}
\item $g$ is symplectic;
\item $\Fix(g)$ contains a curve;
\item $\Fix(g)$ is smooth of pure dimension $1$.
\end{enumerate}
\end{prop}

\begin{proof}
If $g\in \S(X)$, then there exist local coordinates around a fixed points $P$ such that $d_Pg$ acts as
$$(z_1,z_2,z_3)\mapsto (z_1,-z_2,-z_3).$$
The fixed locus has local equation $z_2=z_3=0$, and so is smooth at $P$ and the component containing $P$ has dimension $1$. To complete the proof, one can use Lemma \ref{LEM:LTG} and see that the same description is true near all the fixed points of $g$.
\end{proof}

\noindent Quotients of the form $X/G$, where $G=<g>$ with $g$ a small involution, are of three types and they are described by their fixed locus.
\vspace{2mm}

\begin{prop}
\label{PROP:INVCLASS}
Let $g$ be a small involution on a Calabi-Yau threefold $X$. Call $G$ the cyclic group generated by $g$. Then one of the following holds:
\begin{itemize}
\item $\Fix(g)$ is empty, $g$ is symplectic and $X/G$ is a smooth Calabi-Yau threefold.
\item $\Dim(\Fix(g))=0$, $g$ is not symplectic and $X/G$ has precisely $16$ singular points that are all terminal.
\item $\Dim(\Fix(g))=1$, $g$ is symplectic, $X/G$ is a singular Calabi-Yau threefold whose singular locus has pure dimension $1$.
\end{itemize}
\end{prop}

\begin{proof}
If $g$ has empty fixed locus, the action of $g$ is free, proving the first part.
If the fixed locus has dimension $0$, the eigenvalues of $d_Pg$ are all equal to $-1$ for each $P\in\Fix(g)$. This implies that $g$ isn't symplectic. Moreover, because $(1+1+1) \equiv_2 1$, all the fixed points of $g$ gives terminal singularities on the quotients. By Theorem (\ref{THM:TERM}) the fixed points are $16$ and each of them corresponds to a singular point.
Finally, if $\Fix(g)$ has dimension $1$, $X/G$ is a normal projective threefold with canonical and Gorenstein singularities, i.e., it is a singular Calabi-Yau threefold\footnote{They are interesting because a crepant resulution always exists and it is a smooth Calabi-Yau.}.  
\end{proof}

%


\subsection{Automorphism of order three with isolated fixed points}

Assume that $g$ is an automorphism of order $3$ with isolated fixed points on a Calabi-Yau threefold. Call $\lambda$ a primitive root of unity of order $3$. By Lemma (\ref{LEM:LTG}), for each fixed point $P$, $d_Pg$ has the same determinant. There are three possible case, namely $\det(d_Pg)=1,\lambda$ and $\lambda^2$. Recall that if $P$ is an isolated fixed point; hence $a_i(P)\neq 0$ for all $i$.
\vspace{2mm}

\noindent If $\det(d_Pg)=1$, i.e., if $g$ is symplectic, then $$(\lambda^{a(P)},\lambda^{b(P)},\lambda^{c(P)})\in \{(\lambda,\lambda,\lambda),(\lambda^2,\lambda^2,\lambda^2)\}.$$
Denote by $x_1$ the contribution $C(P)$ of a point such that $(\lambda^a(P),\lambda^b(P),\lambda^c(P))=(\lambda,\lambda,\lambda)$ and $x_2$ the contribution of a point of the other type. Call $n_1$ and $n_2$ the number of such points.
Is easy to see that
$$x_1=\frac{1}{(1-\lambda)^3}=\pm\frac{i\sqrt{3}}{9}$$
and that $\bar{x}_1=x_2$.
The Lefschetz fixed point formula is then $0=n_1x_1+n_2\bar{x}_1.$
Being $x_1$ pure imaginary, one has $n_1=n_2$. In particular, $|\Fix(g)|$ is even.
\vspace{2mm}

\noindent If $\det(d_Pg)=\lambda$ then, for each $P\in \Fix(g)$, $(\lambda^{a(P)},\lambda^{b(P)},\lambda^{c(P)})=(\lambda,\lambda,\lambda^2)$
up to permutations. This implies that every point will give a terminal point on the quotient. So, by Theorem $\ref{THM:TERM}$, the fixed points are $9$. The case for which $\det(d_Pg)$ is similar.
\vspace{2mm}

\noindent We have proved the following Proposition:
\begin{prop}
\label{PROP:AUT3ISO}
Let $g$ be an automorphism of order $3$ on a Calabi-Yau threefold $X$. Call $G$ the cyclic group generated by $g$ and assume that it has a finite numeber of fixed points. Then one of the following holds:
\begin{itemize}
\item $g$ is symplectic and $X/G$ is a singular Calabi-Yau threefold with an even number of singular points.
\item $g$ isn't symplectic, $|\Fix(g)|=9$ and $X/G$ has exactly $9$ singular points. All of them are terminal.
\end{itemize}
\end{prop}


\subsection{Automorphism of order five with isolated fixed points}

Consider now an automorphism $g$ of order $5$ such that $\Fix(g)$ has dimension $0$. We have seen that if $g$ is such that $X/<g>$ has terminal singularities, then $|\Fix(G)|=5$. Now we will show that if $g$ is not symplectic (and it has isolated fixed points) then it has at least $5$ fixed points and the minimum is achieved if and only if $X/<g>$ has terminal singularities.
\vspace{4mm}

\noindent Recall that, given an isolated fixed point $P$, we have defined
$$S_n(P)=\sum_{\substack{0\le k_1,k_2,k_3\le p-1 \\ a_1(P)k_1+a_2(P)k_2+a_3(P)k_3\equiv_p n}} k_1k_2k_3$$
and that, if $g\not\in\S(X)$, we have
$$\sum_{x\in\Fix(g)}\left(\frac{p^3(p-1)^3}{8}-pS_0(x)\right)=p^4$$
by Theorem (\ref{THM:CONTI}). If we define
$$A:=\{( 4, 1, 1 ),
    ( 3, 2, 1 ),
    ( 4, 2, 1 ),
    ( 3, 2, 2 ),
    ( 4, 3, 1 ),
    ( 3, 3, 2 ),
    ( 4, 4, 1 ),
    ( 4, 3, 2 )
\}\mbox{ and }$$
$$B:=\{
    ( 2, 2, 2 ),
    ( 4, 4, 3 ),
    ( 3, 3, 1 ),
    ( 4, 4, 4 ),
    ( 1, 1, 1 ),
    ( 4, 2, 2 ),
    ( 2, 1, 1 ),
    ( 3, 3, 3 )
\},$$
these sets correspond to all the possible values for $(a(P),b(P),c(P))$ for the case $g\not\in\S(X)$.
$A$ is precisely the set for $(a(P)+b(P)+c(P) \mod 5) \in \{a(P),b(P),c(P)\}$, i.e. the set corresponding to $P$ that gives terminal singularities on $X/<g>$.
By direct inspection, we see that $S_0(P)$ is equal to $175$ if and only if $(a(P),b(P),c(P))\in A$. The values that $S_0(P)$ can assume in the other case are $200$ and $225$.
Call $n,q_1,q_2$ the number of points for which $S_0(P)$ is equal respectively to $175,200$ and $225$. In particular, $X/<g>$ has $n+q_1+q_2$ singular points and exactly $n$ are terminal.
The relation in Theorem (\ref{THM:CONTI}) is then
$$p^4=\sum_{x\in\Fix(g)}\left(\frac{p^3(p-1)^3}{8}-pS_0(x)\right)=$$
$$=n\left(\frac{p^3(p-1)^3}{8}-175p\right)+q_1\left(\frac{p^3(p-1)^3}{8}-200p\right)+q_2\left(\frac{p^3(p-1)^3}{8}-225p\right)=$$
$$=(n+q_1+q_2)\left(\frac{p^3(p-1)^3}{8}-175p\right)-25pq_1-50pq_2=(n+q_1+q_2)5^3-5^3q_1-2\cdot 5^3q_2$$
that is
$$5=(n+q_1+q_2)-q_1-2q_2=n-q_2 \Longrightarrow n=5+q_2.$$
This implies that the number of fixed points is $5+q_1+2q_2$. In particular it is at least $5$ and, moreover, it is equal to $5$ if and only if $q_1=q_2=0$.

\begin{prop}
\label{PROP:AUT5ISO}
Let $g$ be a non symplectic automorphism of order $5$ on a Calabi-Yau threefold $X$. Assume that it has a finite number of fixed points and call $G$ the cyclic group generated by $g$. Then one of the following holds:
\begin{itemize}
\item $|\Fix(G)|=5$ and $X/G$ has only terminal singularities.
\item $|\Fix(G)|>5$, $X/G$ has $5$ or more terminal singularities and at least another fixed point.
\end{itemize}
\end{prop}


\section{Some examples}
\label{SE:EX}

\subsection{Quotient with terminal singularities}
\noindent Here we will construct quotients of Calabi-Yau threefolds with only terminal singularities with respect to $\Z_3,\Z_5,\Z_2$ and $\Z_2\oplus\Z_2$.

\begin{ex}
Let $X$ be $\P^2\times \P^2$ with projective coordinates $x_i$ and $y_i$ on the two factors of $X$. Consider the automorphism of $X$ defined by
$$g((x_0:x_1:x_2)\times(y_0:y_1:y_2)):=(x_0:x_1:\lambda x_2)\times(y_0:\lambda y_1:\lambda^2 y_2)$$
with $\lambda$ a primitive root of unity of order $3$.
It is easy to see that $g$ has order $3$ and that its fixed locus has $6$ irreducible components, three of which are rational curves. More precisely, if $S:=\{(1:0:0),(0:1:0),(0:0:1)\}$, then
$$\Fix(g)=(\{(x_0:x_1:0) | (x_0:x_1)\in\P^1\}\cup\{(0:0:1)\})\times S.$$
If $G:=<g>$, define $V$ to be the vector space $\H^0(X,-K_X)^{G}$ of all $G-$invariants anticanonical sections. By direct computations, we see that the generic element of $V$ is smooth (because $V$ has empty base locus) and doesn't vanish on any of isolated points of $\Fix(g)$. The generic element gives an invariant Calabi-Yau $Y$ that intersect only the components of $\Fix(g)$ having dimension $1$. For $Y$ generic, none of these three curves is contained in $Y$, so this implies that $Y$ meets each one in exactly $3$ points (by simple calculation of intersection theory). In total there are $9$ fixed points of $g$ that are on $Y$. By our discussion on the number of fixed points of an automorphism of order three on a Calabi-Yau threefold, we can conclude that $Y/G=Y/\Z_3$ has exactly $9$ singular points and that each of them are terminal.
\end{ex}

\begin{ex}
Let $X$ be $\P^4$ with projective coordinates $x_i$. Take $g$ to be the automorphism of $X$ defined by
$$g((x_0:x_1:x_2:x_3:x_4):=(x_0:x_1:\lambda x_2:\lambda^2 x_3:\lambda^3 x_4).$$
with $\lambda^5=1$ and primitive. Call $G$ the group generated by $g$ and $Y$ the Fermat hypersurface of degree five in $\P^4$.
$Y$ is a smooth Calabi-Yau threefold and is easily seen to be invariant with respect to $G$. The fixed locus of $g$ on $X$ is
$$\Fix(g)=\{(x_0:x_1:0:0:0) | (x_0:x_1)\in\P^1\}\cup$$
$$\cup \{(0:0:1:0:0),(0:0:0:1:0),(0:0:0:0:1)\}$$
and $Y$ doesn't meet the isolated points of $\Fix(g)$. The intersection of $Y$ with $\Fix(g)$ are $5$ isolated points.
If one consider the quotient $Y/G$, because $G$ is small, one has that $\Sing(Y/G)=\Fix(g)$. In particular $Y/G=Y/\Z_5$ has exactly $5$ isolated fixed points and each of them is terminal.
\end{ex}

\begin{ex}
\noindent Set $X=\P^1\times \P^1\times \P^1\times \P^1$ with $(x_i:y_i)$ that are projective coordinates on the $i-$th $\P^1$. In \cite{BF} and \cite{BFNP} the authors study the automorphisms of Calabi-Yau manifolds embedded in $X$ that have empty fixed locus. The authors produce a classification of all the admissible pairs in $X$, i.e. the pairs $(Y,G)$ where $Y$ is a Calabi-Yau threefold and $G$ is a finite group of authomorphisms of $X$ that stabilizes $Y$ and acts freely on $Y$. Here we will show that one can easily constuct examples with a different kind of fixed points.
\vspace{2mm}

\noindent Every $g\in \Aut(X)$ acts on the $4$ factors (see, for instance, \cite{BF}) giving a  surjective homomorphism
$\pi\colon \Aut(X) \rightarrow S_4$ with kernel $\PGL(2)^{\times 4}$. On the other hand the permutations of the 
factors give an inclusion $S_4 \hookrightarrow \Aut(X)$ splitting $\pi$ and therefore giving a structure of semidirect product
$$
\Aut(X) \cong S_4 \ltimes \PGL(2)^{\times 4} 
$$
Concretely this gives, $\forall g \in \Aut X$, a unique decomposition $g=(A_i)\circ \sigma$ where $\sigma=\pi(g)$ and  
$(A_i)=(A_1,A_2,A_3,A_4)\in \PGL(2)^{\times 4}$. Denote with $A$ and $B$ the automorphisms of $\P^1$ that send $(x_1:y_1)$ respectively in
$(x_1:-y_1)$ and $(y_1:x_1)$.
\vspace{2mm}

\noindent Call $g:=(\id,\id,A,A)\circ (12)$ and $h:=(A,A,\id,\id)\circ (34)$. It is easy to see that $G:=<g,h>\simeq \Z_2\oplus\Z_2$. As an automorphism of $X$, $g$ have fixed locus composed of $4$ rational curves. These are 
$$\{(P,P,Q_1,Q_2) \,|\, P\in\P^1,\, Q_1,Q_2\in\{(1:0),(0:1)\}\}.$$
Call $C$ a component of $\Fix(g)$ and consider a generic element $Y\in |-K_X|$.
It is easy to see that $Y\cdot C =4$ so generically one expect that $Y\cap C$ has $4$ fixed points. A similar result holds for $h$. For $gh=(A,A,A,A)\circ (12)(34)$ 
one can see that $\Fix(gh)$ is composed of isolated points so generically one expect that the general member of $|-K_X|$ doesn't meet $\Fix(gh)$.
Call $V$ the vector space $\H^0(X,-K_X)^{G}$ of all $G-$invariants anticanonical sections. By direct computations one see that $V$ has empty base locus so the generic element is a smooth Calabi-Yau that admits an action of $G$.
\vspace{2mm}

\noindent We point out that the Calabi-Yau $Y$ constructed as zero locus of the generic section of $V$ doesn't contain $Fix(g)$. Then $\Fix(g)$ meets $Y$ at isolated points. By Theorem \ref{THM:TERM}, $g$ is an involution with exactly $16$ fixed points and acts as $-1$ on $\H^{0,3}(X)$. The same is true for $h$ and $gh$ which have to act as $\id$ on $\H^{0,3}(X)$. Because $gh$ is an involution and $(gh)^*|_{H^{0,3}(X)}=\id$ one has that $\Fix(gh)$ is either empty or it has pure dimension $1$. The latter cannot occour because $gh$ has a finite number of fixed points on $X$ so $gh$ acts freely on $Y$.
\vspace{2mm}

\noindent By direct computation the common fixed points of $g$ and $h$ on $X$ are $4$, namely:
$$\{(P,P,Q,Q) \,|\, P,Q\in\{(1:0),(0:1)\}\}.$$
The generic invariant section isn't zero on any of these point so, for the generic invariant Calabi-Yau $Y$, the fixed locus of $g$ and the one of $h$ are disjoint.
\vspace{2mm}

\noindent To conclude, for $Y$ generic, we have given four quotients, namely: $Z_1=Y/<g>$, $Z_2=Y/<h>$, $Z_3=Y/<gh>$ and $Z_4=Y/G=Y/(\Z_2\oplus\Z_2)$. $Z_3$ is smooth because $gh$ acts freely on $Y$; thus the map $Y\rightarrow Y/<gh>$ is an \'etale cover of degree $2$. $Z_1$, $Z_2$ and $Z_4$ have only terminal singularities by construction. $Z_1$ and $Z_2$ have exactly $16$ singular points. To compute the number of singular points of $Z_4$ we can use Burnside's Lemma on the action of $G$ restricted on $\Fix(G)$. In fact $\Sing(Y/G)=\Fix(G)/G$ because $G$ is small. The formula for this case is
$$|\Sing(Y/G)|=|\Fix(G)/G|=\frac{1}{|G|}\sum_{g \in G}|\Fix(g)|=$$
$$=\frac{1}{|G|}(\Fix(\id)+\Fix(g)+\Fix(h)+\Fix(gh))=\frac{1}{4}\left(32+16+16+0\right)=16.$$
In conclusion $Z_4$, like $Z_1$ and $Z_2$, has exactly $16$ isolated and terminal singularities.
\vspace{2mm}

\noindent Another way to see this fact is to note that $<gh>$ is normal in $G$ with quotient generated by an involution which we will denote $\hat{g}$. We can write the quotient $Z_4=Y/G$ as $(Y/<gh>)/(G/<gh>)=Z_3/(G/<gh>)$. Hence, $Y/G$ can be seen as the quotient of a smooth Calabi-Yau threefold by the single involution $\hat{g}$, an involution that has isolated fixed points. By our classification of fixed locus of involutions on Calabi-Yau manifolds, $\hat{g}$ has exactly $16$ fixed points. Since we have
$$\Sing(Y/G)=\Fix(\hat{g})/<\hat{g}>=\Fix(\hat{k}),$$
we obtain again that $Z_4$ has $16$ fixed points.

\end{ex}

\subsection{A quotient with non-isolated Gorenstein singularities}
Here we will construct a quotient of a Calabi-Yau by a group of order $32$ that is contained in $S(X)$.

\begin{ex}
\noindent Use the same notation introduced in the last example. Call 
$$g:=(\id,\id,\id,A)\circ (1324),h:=(B,B,B,B), k:=(14)(23)$$ and consider the group $G:=<g,h,k>$.
$g,h$ and $k$ satisfy the following relations
$$g^8=h^2=k^2=1\quad gh=hg\quad kh=hk\quad gk=kg^{-1}.$$
Using these relations it is easy to see that $G$ has $32$ elements and that $G$ is isomorphic to $D_{16}\times \Z_2$ where $D_{16}=<g,k>$ is the dihedral group with $16$ elements and $\Z_2=<h>$. The elements of $G$ can be written uniquely as $g^ah^bk^c$ with $0\leq a\leq 7$ and $0\leq b,c\leq 1$. All the elements with $c=1$ are involutions.
Considering all the elements as automorphisms acting on $X$, $\Fix(g^ah^b)$ has a finite number of elements whereas $\Fix(k)$ has pure dimension $2$ and the same is true for  $\Fix(g^ah^bk)$. 
If we denote by $V:=\H^0(X,K_X)^G$ it can be seen that the Calabi-Yau $Y$ given by the generic $s\in V$, i.e., the generic $G-$invariant Calabi-Yau is smooth and satisfies
$$\Fix(g^ah^b)\cap Y = \emptyset\quad\quad \Dim(\Fix(k)\cap Y)=1$$
for all $a,b$. the involution $k$ acts on $Y$ with a fixed locus of dimension $1$. By Proposition \ref{PROP:INVCLASS} this implies that $k\in\S(Y)$. The two elements $g$ and $h$ do not have fixed points on $Y$ so they are elements of $\S(Y)$. This is enough to conclude that $G\leq \Gor(Y)$. It can be shown that $\Aut(Y)=\S(Y)$ for the generic $G-$invariant Calabi-Yau threefold.
\vspace{2mm}

\noindent By the Moishezon-Nakai criterion, being $Y \in |-K_X|$ ample, $Y$ cannot be disjoint from $\Fix(g^ah^bk)$ that has dimension $2$ on $X$. We know that $g^ah^bk$ is a symplectic involution so we conclude that $\Fix(g^ah^bk)\cap Y$ is smooth of pure dimension $1$ for each $a,b$.


All the irreducible components of $\Fix(G)$ are obtained from the irreducible components of $\Fix(k), \Fix(gk), \Fix(hk)$ and $\Fix(ghk)$ using $\Fix(b^{-1}ab)=b(\Fix(a))$ and $x^{-1}(g^ah^bk)x=g^{a+2d}h^bk$ for $x\in G$. Using MAGMA, it is possible to check that there are at least $32$ components and some of them are rational curves. Starting from these remarks we can argue that there are at least $4$ orbits for the action of $G$ on $\Fix(G)$.
\vspace{2mm}

\noindent It is interesting to note that $H:=<g,h>\simeq \Z_8\times \Z_2$ is a subgroup of $G$ of index $2$ and thus is a normal subgroup of $G$. The quotient $Y/G$ can be viewed as $(Y/H)/(G/H)$. It is interesting to write the quotient like this because $G/H\simeq \Z_2$ and it is generated by an involution which we will denote $\hat{k}$. Moreover $\Fix(H)$ is empty (the pair $(Y,H)$ is one of the admissible pairs studied in \cite{BFNP}) so $Y\rightarrow Y/H$ is an \`etale cover of degree $16$ and $Y/H$ is again a smooth Calabi-Yau threefold. Hence $Y/G$ may be viewed as the quotient of a smooth Calabi-Yau threefold by the single involution $\hat{k}$. The fixed locus of $\hat{k}$ satisfies
$\Fix(\hat{k})=\Fix(\hat{k})/<k>=\Sing(Y/G)=\Fix(G)/G$, so the singular locus of $Y/G$ has at least $4$ irreducible components and at least one of them is a rational curve. 
\vspace{2mm}

\end{ex}


\end{document}